\documentclass[12pt,dvipdfmx]{article}
\usepackage[margin=1truein]{geometry}
\usepackage[dvipdfmx]{graphicx}
\usepackage{subfigure}
\usepackage{latexsym, amsmath, amssymb, amsthm, a4, epsfig}
\usepackage{mathtools}
\usepackage{color}
\usepackage{tikz}
\usepackage{ascmac}
\usepackage{bm}
\usepackage{bbm}
\usepackage{enumerate}
\usepackage{enumitem}
\usetikzlibrary{intersections, calc, arrows, positioning, arrows.meta}

\allowdisplaybreaks
\newcommand{\ceq}{\coloneqq}
\newcommand{\dd}{\mathrm{d}}
\newcommand{\E}{\mathrm{E}}
\newcommand{\ds}{\displaystyle}
\newcommand{\Poi}{\mathrm{Poi}}
\newcommand{\orig}{\mathrm{orig}}

\newcommand{\join}{\mathrm{join}}
\newcommand{\balk}{\mathrm{balk}}

\newcommand{\red}[1]{{#1}}

\theoremstyle{definition}
\newtheorem{theorem}{Theorem}

\newtheorem{lemma}[theorem]{Lemma}
\newtheorem{remark}[theorem]{Remark}

\begin{document}

\title{\textbf{Computing transient probabilities in Markovian queues 
with balking conditional on a fixed number of joined customers}}
\author{Kaito Hayashi, Yoshiaki Inoue, and Tetsuya Takine\thanks{
K.\ Hayashi, Y.\ Inoue, and T.\ Takine are with
Department of Information and Communications Technology, 
The University of Osaka, Japan.
\\[2ex]
E-mail: hayashi23@post.comm.eng.osaka-u.ac.jp,\,
\{yoshiaki, takine\}@comm.eng.osaka-u.ac.jp
}}
\date{}
\maketitle

\begin{abstract}
We analyze the transient behavior of a Markovian queue with balking,
conditional on exactly $K$ customers joining the system in a finite
time interval $[0, T]$. A key quantity of interest is the cumulative
number of balking customers, for which we consider the probability
generating function (PGF) and derive an equation it satisfies. This
approach circumvents the computational burden arising from dealing
with high-dimensional Markovian system, which arises
when attempting to directly compute the joint distribution of
\red{%
the cumulative number of balking customers together with the total 
number of joining customers and the number of customers in the system.}
To compute state transitions in $(t, T]$ efficiently,
we examine the conditional state probability at time $T$ given the
system state at time $T - u$, and show that it satisfies a linear
differential equation in $u$.
For piecewise-constant arrival rates, we develop a numerical procedure
\red{%
for evaluating the moment of the total number of balking customers, 
jointly with the cumulative number of joining customers and the 
number of customers in the system, under the condition
that $K$ customers joined.}
We also present numerical examples that highlight
counterintuitive behaviors arising from this conditioning, along with
explanations of the underlying mechanisms.
\\[5pt]
\end{abstract}

\section{Introduction}
In many real-world queueing systems, such as waiting lines in
cafeterias or at service counters in retail stores and offices,
congestion can cause arriving customers to leave without receiving
service. This general behavior can take two forms: when customers give
up entering the system upon arrival, it is referred to as balking,
while when they enter but leave before receiving service due to
excessive waiting, it is referred to as reneging.

When customers abandon service opportunities through balking or
reneging, the system loses potential revenue or throughput.
Consequently, one of the primary performance metrics in such systems
is the number of customers who arrive but do not receive service.
Since these behaviors are typically influenced by factors that reflect
the level of congestion, such as the number of customers already in
the system or the expected waiting time, the congestion level at each
time point also serves as an important measure.

Various queueing models incorporating balking and reneging have been
studied extensively. Early works include
\cite{Ancker1963-a,Ancker1963-b,Barrer1957-a,Barrer1957-b,Haight1957,Haight1959},
which investigate customer abandonment in different settings. Reneging
behavior under both random and first-come, first-served service
disciplines is examined in \cite{Barrer1957-a, Barrer1957-b}. Balking
behavior is modeled in \cite{Haight1957}, where stochastic thresholds
are considered to represent each customer's maximum acceptable number
of customers in the system on arrival. This is extended in
\cite{Haight1959} to incorporate reneging based on individual waiting
time tolerances. In \cite{Ancker1963-a,Ancker1963-b}, probabilistic
decisions to join the system based on the observed queue length are
considered, where stationary distributions of several key system
metrics are derived.

Subsequent studies have examined time-dependent arrival rates in
queueing models. Fluid limits of queues with arrival rates depending
on both time and congestion are analyzed in \cite{Whitt2006}. A
discrete-time approximation to Poisson arrivals influenced by time
and system state is adopted in \cite{Chassioti2014}. For a
comprehensive survey of queueing models with customer abandonment, see
\cite{Sharma2023}.

In this paper, we consider a Markovian queueing system with balking,
under a nonstandard assumption that \textit{exactly $K$ customers join the
system in a finite time interval $[0, T]$}. This conditional
framework departs from conventional models based on unconditioned
arrival processes, and reflects realistic situations in which the total
number of customers joining the system in a fixed period (such as one
business day) is directly observable as a summary statistic. 
Although $K$ is not a controllable parameter of the system,
it is often used as a reference quantity in performance evaluation and
planning. Analyzing system behavior conditional on this observable
quantity thus provides a complementary perspective to traditional
time-averaged and transient analyses.

We consider a continuous-time Markovian queue 
with arrivals following a non-homogeneous Poisson process,
where the system state at each time is given by the cumulative
number of customers who have joined, the number of customers in the system, and
a service phase, incorporating various service mechanisms.
Among various performance metrics in this setting,
a key quantity of interest is the cumulative number of balking
customers. 
\red{%
However, the computation of this quantity as a Markovian counting 
process requires analyzing a large underlying state-space Markov 
chain, leading to a substantial increase in computational complexity.}

To obtain a computationally tractable characterization, we analyze the
probability generating function (PGF) of the cumulative number of
balking customers, and derive the corresponding factorial moments.
We also consider the conditional state probability of the system 
at time $T$ given its state at time $T - u$, and show that it satisfies a linear
differential equation in $u$. This result enables us to efficiently
compute factorial moments of the cumulative number of balking 
customers conditional on the total number of joined customers.

Furthermore, for piecewise-constant arrival rates, we develop a numerical procedure
for evaluating the factorial moments \red{of the total number of balking customers} under 
the condition that $K$ customers joined. 
We present numerical examples to illustrate how the 
\red{system behavior} under such conditioning can exhibit
counterintuitive features, such as persistent time-dependent dynamics
that arise even under time-homogeneous arrival rates, and we provide 
explanations for these dynamics.

We briefly review
prior research on queueing models conditional on a fixed number of
arrivals \cite{Bet2019, Bet2020, Hayashi2026, Honnappa2015, Louchard1988, Louchard1994, Mandjes2025, Minh1977}.
In such models, the arrival times of $K$ customers are assumed to 
be independent and identically distributed (i.i.d.) in a finite interval
$[0, T]$. This setup can be interpreted as conditioning a
nonhomogeneous Poisson process to generate $K$ arrivals in $[0, T]$. 
Although balking is not considered in these models,
several results have been established. In \cite{Minh1977}, a
discrete-time model is analyzed, leading to equations for the PGFs of
the number of customers in the system and the virtual waiting time.
Continuous-time analogs are studied in \cite{Louchard1988,Louchard1994}, 
where it is shown that the queue-length and workload processes
converge weakly to Gaussian
processes as $K$ increases. Approximation techniques have also been developed; for
instance, fluid and diffusion limits are derived in \cite{Honnappa2015},
while reflected Brownian limits in critically loaded regimes
are analyzed in \cite{Bet2019, Bet2020}. 
Furthermore, exact characterizations of the 
time-dependent workload and queue-length have been provided in
\cite{Hayashi2026} and \cite{Mandjes2025}.

These studies demonstrate the analytical richness of fixing 
the total number of arrivals, even without balking. In our setting,
balking introduces additional complexity, particularly in computing the
cumulative number of balked customers, as discussed later.

The remainder of this paper is structured as follows.
In Section 2, we describe the model considered 
in this paper, explain the computational difficulties 
encountered with a naive analytical approach, 
and outline our approach.
In Section 3, we present the analysis of
the PGF of the cumulative number of balked customers 
conditional on the total number of arrivals, along 
with \red{the} corresponding results for the factorial moments.
In Section 4, we provide a detailed construction of
a computation algorithm for transient states probabilities 
in piecewise time-homogeneous systems,
and in Section 5, we present numerical examples.
Finally, we conclude the paper in Section 6.

\section{Model and analytical method}
\subsection{\textit{Model}}
We consider a queueing system in which customers arrive 
in a finite time interval $[0,T]$ ($T \in (0,\infty)$). 
The arrival process of customers is assumed to follow 
a nonhomogeneous Poisson process with rate $\lambda(t)$, 
and for $t > T$, we set $\lambda(t) = 0$.
Additionally, arriving customers may balk
based on the number of customers present just before 
their arrival. 
Specifically, customers who see $\ell$ customers 
($\ell = 0,1,\ldots$) present on arrival decides to 
join the system with probability $\beta_\ell$ 
($0 \leq \beta_\ell \leq 1$) and they 
leave immediately with probability $1-\beta_\ell$.

Let $A(t)$ ($t \geq 0$) denote the cumulative number of 
customers who have arrived at the system in the time 
interval $[0,t]$. 
Among these arriving customers, let $A^\join(t)$ 
denote the cumulative number of customers who 
have joined the system, and let $A^\balk(t)$ denote 
the cumulative number of customers who have balked. 
Let $D(t)$ denote the cumulative number of customers who 
have completed service and departed from the system in 
the time interval $[0,t]$. Also let $L(t)$ 
denote the number of customers in the system at time $t$. 
For simplicity, we assume that the system is initially empty, 
i.e., $L(0)=0$. 
By definition, we have the following relations:
\begin{align*}
A(t) &= A^\join(t) + A^\balk(t), \quad t \geq 0,
\\
A^\join(t) &= L(t) + D(t), \quad t \geq 0.
\end{align*}

We assume that the service mechanism of the system is 
Markovian, and we define $S(t)$ as the service phase at time 
$t$. The service phase $S(t)$ is assumed 
to contain sufficient information to describe the 
behavior of the system and it takes values in a finite 
set $\mathcal{S}$. 
Consequently, the triplet $(A^\join(t), L(t), S(t))_{t \geq 0}$ 
forms a continuous-time Markov chain defined on the 
state space $\Omega^\orig$, given by 
\[
\Omega^\orig
\ceq
\{(k,\ell,s);\ 
k=0,1,\ldots,\ 
\ell=0,1,\ldots,k,\ 
s\in\mathcal{S}\}.
\]
The transition rate matrix of this continuous-time 
Markov chain is denoted by $\bm{Q}^\orig(t)$, 
where states are assumed to be ordered lexicographically. 
With this ordering, we regard $\bm{Q}^\orig(t)$ as 
the transition rate matrix of a bivariate Markov chain
with level variable $A^\join(t)$ and phase variable $(L(t), S(t))$.

In this paper, we conduct a time-dependent analysis of 
the system conditional on the event that the cumulative 
number of joined customers by time $T$ equals a fixed constant $K$ ($K=1,2,\ldots$). 
Therefore, it suffices to consider only states 
where the cumulative number of joined customers does not exceed $K$.
Accordingly, we define the restricted state space 
$\Omega$ as 
\[
\Omega
\ceq
\{(k,\ell,s);\ 
k=0,1,\ldots,K,\ 
\ell=0,1,\ldots,k,\ 
s\in\mathcal{S}\}.
\]
Under this condition, the transition rate matrix
$\bm{Q}^\orig(t)$ of the continuous-time Markov chain 
$(A^\join(t), L(t), S(t))_{t \geq 0}$ can be expressed 
using the transition rate matrix $\bm{Q}(t)$ 
for transitions between states belonging to $\Omega$: 
\begin{equation*}
\bm{Q}^\orig(t)
=
\left[
\begin{array}{cc}
\bm{Q}(t) & \bm{Q}_{1,2}^\orig(t)\\
\bm{O} & \bm{Q}_{2,2}^\orig(t)
\end{array}
\right].
% \label{eq:transition rate matrix}
\end{equation*}

Let $p_{k,\ell,s}(t)$ denote the state probability at time $t$:
\begin{equation*}
p_{k,\ell,s}(t)
\ceq
\Pr[A^\join(t)=k, L(t)=\ell, S(t)=s],
\quad
(k,\ell,s)\in\Omega.
% \label{eq:p-def}
\end{equation*}
Let $|\Omega|$ denote the size of the state space.
We define $\bm{p}(t)$ as a $1 \times |\Omega|$ vector 
consisting of the elements $p_{k,\ell,s}(t)$, 
arranged in lexicographic order.
By definition, we have for $t=0$, 
\begin{equation}
p_{k,\ell,s}(0)
=
\left\{
\begin{aligned}
&\Pr[S(0)=s],\quad && k = 0,\ \ell = 0, s \in \mathcal{S},\\
&0, \quad && \text{otherwise}.
\end{aligned}
\right.
\label{eq:p(0)}
\end{equation}
Under this condition, $\bm{p}(t)$ is given by
the solution of the following differential equation:
\begin{equation}
\frac{\dd \bm{p}(t)}{\dd t}
=
\bm{p}(t)\bm{Q}(t).
\label{eq:p-differential equation}
\end{equation}

In this paper, we consider the time-dependent behavior of the system,
conditional on exactly $K$ customers having joined before time $T$.
Specifically, for each $t \in [0, T]$, we examine the joint
probability of the cumulative number of joined customers, the number
of customers in the system, the service phase, and the cumulative
number of balked customers.
We define this time-dependent conditional joint probability 
as $\pi_{k,\ell,s,b}(t \mid K)$:
\begin{equation}
\pi_{k,\ell,s,b}(t \mid K)
\ceq
\Pr[A^\join(t)=k, L(t)=\ell, S(t)=s, A^\balk(t)=b \mid A^\join(T)=K].
\label{eq:pi-def}
\end{equation}

\subsection{\textit{Analytical method}}

A straightforward approach to compute $\pi_{k,\ell,s,b}(t \mid K)$ 
is to analyze a continuous-time Markov chain defined on an
expanded state space
\begin{equation}
\breve{\Omega}
\ceq
\{(k,\ell,s,b);\ 
k = 0,1,\ldots,K,\ 
\ell = 0,1,\ldots,k,\ 
s \in \mathcal{S},\ 
b = 0,1,\ldots\},
\label{eq:breveOmega}
\end{equation}
where $k$, $\ell$, $s$, and $b$
represent the number of joined customers,
the number of customers in the system, 
the service phase, and the number of balked customers,
respectively.
Let $p_{k,\ell,s,b}(t)$ denote the state probability of the expanded Markov
chain defined on $\breve{\Omega}$:
\[
p_{k,\ell,s,b}(t)
\ceq
\Pr[A^\join(t)=k, L(t)=\ell, S(t)=s, A^\balk(t)=b],
\quad
(k,\ell,s,b) \in \breve{\Omega}.
\]
We also define $\breve{\bm{p}}(t)$ as the corresponding probability vector. 
Similarly to (\ref{eq:p-differential equation}),
the state probability vector $\breve{\bm{p}}(t)$ satisfies
\begin{equation}
\frac{\dd \breve{\bm{p}}(t)}{\dd t}
=
\breve{\bm{p}}(t) \breve{\bm{Q}}(t), 
\label{eq:p-straightforward approach}
\end{equation}
where $\breve{\bm{Q}}(t)$ denotes the transition rate matrix 
of the Markov chain defined on $\breve{\Omega}$.

We can then rewrite \eqref{eq:pi-def} as
\begin{align}
\lefteqn{
\pi_{k,\ell,s,b}(t \mid K)
}
\nonumber
\\
&=
\frac{%
\Pr[A^\join(t)=k, L(t)=\ell, S(t)=s, A^\balk(t)=b,
A^\join(T)=K]}
{\Pr[A^\join(T)=K]}
\nonumber
\\
&=
\frac{%
p_{k,\ell,s,b}(t)
\Pr[A^\join(T)=K \mid A^\join(t)=k, L(t)=\ell, S(t)=s, A^\balk(t)=b]
}
{\Pr[A^\join(T)=K]}
\nonumber
\\
&=
\frac{%
p_{k,\ell,s,b}(t)
\Pr[A^\join(T)=K \mid A^\join(t)=k, L(t)=\ell, S(t)=s]}
{\Pr[A^\join(T)=K]},
\label{eq:pi-straightforward-approach}
\end{align}
where the last equation holds because 
$A^\join(T)$ and $A^\balk(t)$ are conditionally independent 
given the system state $(A^\join(t), L(t), S(t))$.
Since the transition probability $\Pr[A^\join(T)=K \mid A^\join(t)=k,
L(t)=\ell, S(t)=s]$ is also determined by the transition rate
matrix $\breve{\bm{Q}}(t)$, the expression (\ref{eq:pi-straightforward-approach}) 
fully characterizes $\pi_{k,\ell,s,b}(t \mid K)$.

As an example, consider a piecewise time-homogeneous system
with 
\[
\breve{\bm{Q}}(t) = \breve{\bm{Q}}_n,
\quad
T_{n-1} < t \leq T_n,
\]
for $T_0 = 0 \leq T_1 \leq T_2 \leq \cdots \leq T_{N-1} \leq T_N=T$.
The solution of \eqref{eq:p-straightforward approach} in this case is
given by the following equation:
\begin{equation}
\breve{\bm{p}}(t)
=
\breve{\bm{p}}(0)
\left(
\prod_{i=1}^{n-1}
\exp\bigl(\breve{\bm{Q}}_i \cdot (T_i - T_{i-1})\bigr)
\right)
\exp\bigl(\breve{\bm{Q}}_\red{n} \cdot (t - T_{n-1})\bigr),
\;\;
T_{n-1} < t \leq T_n.
\label{eq:p-straightforward-approach-TH}
\end{equation}
Therefore, we obtain $\pi_{k,\ell,s,b}(t \mid K)$ from 
\eqref{eq:pi-straightforward-approach}
with $p_{k,\ell,s,b}(t)$ (the elements of $\breve{\bm{p}}(t)$)
given by \eqref{eq:p-straightforward-approach-TH}.

However, there are two main challenges in computing $\pi_{k,\ell,s,b}(t \mid K)$ using this approach:
\begin{enumerate}[label=(\roman*)]
\item Accounting for the cumulative number of balked customers.
\begin{itemize}[leftmargin=1em]
\item[]
\red{%
Since the Markov chain is defined on a countably infinite state 
space $\breve{\Omega}$, it is difficult to directly compute 
the transient state probabilities $\breve{\bm{p}}(t)$ using \eqref{eq:p-straightforward-approach-TH}.}

One possible workaround is to treat $(A^\balk(t))_{t \geq 0}$ as a
Markovian counting process driven by the underlying process
$(A^\join(t), L(t), S(t)) \in \Omega$, for which an efficient
computational method is known in the literature \cite{Lucantoni1985, Takine1994}.
However, the state space $\Omega$ still grows as $|\Omega| =
O(K^2|\mathcal{S}|)$, which remains computationally expensive. For
instance, even with the modest values $K = 100$ and $|\mathcal{S}| = 2$, 
$|\Omega|$ already exceeds $10,\!000$, highlighting the difficulty
of scaling to larger $K$.
\end{itemize}
\item Computing the transition probability 
$\Pr[A^\join(T)=K \mid A^\join(t)=k, L(t)=\ell, S(t)=s]$ in \eqref{eq:pi-straightforward-approach}.
\begin{itemize}[leftmargin=1em]
\item[] Although this transition probability is straightforward to evaluate in principle, 
it imposes a heavy computational burden when $\pi_{k,\ell,s,b}(t \mid K)$ must be
computed across many time points $t$, due to repeated evaluations of this quantity.
\end{itemize}
\end{enumerate}

To address challenge (i), we adopt an approach based on PGFs.
While direct computation of the time-dependent probability $p_{k,\ell,s,b}(t)$ is computationally intensive, the use of PGFs enables efficient calculation of the moments of the cumulative number of balked customers.
More specifically, we define $p_{k,\ell,s}^*(z,t)$ 
($|z| \leq 1$, $t \in [0,T]$) as the (unconditional) PGF of the 
cumulative number of of balked customers, and 
$p_{k,\ell,s}^{(i)}(t)$ ($i = 0,1,\ldots$ $t \in [0,T]$) 
as the $i$-th (unconditional) factorial moment of the 
cumulative number of balked customers:
\begin{align}
p_{k,\ell,s}^*(z,t)
&\ceq
\sum_{b=0}^{\infty}p_{k,\ell,s,b}(t) z^b,
\quad
(k,\ell,s) \in \Omega,\ 0 \leq t \leq T,
\notag
% \label{eq:p^*-def}
\\
p_{k,\ell,s}^{(0)}(t)
&\ceq
p_{k,\ell,s}^*(1,t)
=
p_{k,\ell,s}(t),
\label{eq:p^(0)-def}
\\
p_{k,\ell,s}^{(i)}(t) 
&\ceq
\frac{\partial^i p_{k,\ell,s}^*(z,t)}{\partial z^i}
\bigg|_{z=1}
\notag \\
&=
\E\left[
\prod_{h=0}^{i-1}(A^\balk(t)-h)
\cdot
\mathbbm{1}_{\{A^\join(t)=k, L(t)=\ell, S(t)=s\}}
\right],
\;\;
i = 1,2,\ldots,
\label{eq:p^(i)-def}
\end{align}
where $\mathbbm{1}_{\{\cdot\}}$ denotes the indicator function.
Similarly, we define the PGF and factorial moment of the 
number of balked customers conditional on exactly $K$ 
customers having joined before time $T$:
\begin{align}
\pi_{k,\ell,s}^*(z,t \mid K)
&\ceq
\sum_{b=0}^{\infty}\pi_{k,\ell,s,b}(t \mid K)z^b,
\quad 
(k,\ell,s) \in \Omega,\ 0 \leq t \leq T,
\label{eq:pi^*-def}
\\
\pi_{k,\ell,s}^{(0)}(t \mid K) 
&\ceq 
\pi_{k,\ell,s}^*(1,t \mid K)
\notag\\
&\: =
\Pr[A^\join(t)=k, L(t)=\ell, S(t)=s \mid A^\join(t)=K],
\notag 
\\
\pi_{k,\ell,s}^{(i)}(t \mid K)
&\ceq
\frac{\partial^i \pi_{k,\ell,s}^*(z,t \mid K)}{\partial z^i}
\bigg|_{z=1}
\notag \\
&=
\E \left[
\prod_{h=0}^{i-1} (A^\balk(t)-h) 
\cdot
\mathbbm{1}_{\{A^\join(t)=k, L(t)=\ell, S(t)=s\}}
\mid
A^\join(T)=K
\right],
\notag \\
&\hspace{23em}
i = 1,2,\ldots.
\notag 
\end{align}
We show that the unconditional PGF $p_{k,\ell,s}^*(z,t)$ (for each
$|z| \leq 1$) satisfies a first-order differential equation in time
$t$, and characterize $\pi_{k,\ell,s}^*(z,t \mid K)$ accordingly. 
Furthermore, using these results, we characterize both the 
unconditional factorial moments $p_{k,\ell,s}^{(i)}(t)$ 
and the conditional factorial moments $\pi_{k,\ell,s}^{(i)}(t \mid K)$.
Based on these results, we develop a practical computational procedure 
applicable to piecewise time-homogeneous systems.

To address challenge (ii), we make use of the fact that the cumulative
number of joined customers is fixed at time $T$, and consider the
probability that $A^\join(T) = K$ given the system state at an earlier
time $T - u$ for $0 \leq u \leq T$. This backward-in-time formulation
allows us to efficiently compute the transition probabilities
$\Pr[A^\join(T)=K \mid A^\join(T-u)=k, L(T-u)=\ell, S(T-u)=s]$ for
multiple values of $u$, as they satisfy a linear
differential equation in $u$.

In the following sections, we analyze the conditional PGF
$\pi_{k,\ell,s}^*(z,t \mid K)$ and develop an efficient method for computing it.

\section{Analysis of the time-dependent PGF conditional on $K$ joined customers}

We first consider the unconditional PGF.
Let $\bm{p}^*(z,t)$ denote the row vector consisting of the elements
$p_{k,\ell,s}^*(z,t)$, arranged in lexicographic order.
\begin{lemma}\label{lem:p^*}
For each $|z| \leq 1$, the unconditional PGF $\bm{p}^*(z,t)$ is
determined by the following linear differential equation:
\begin{align}
\bm{p}^*(z,0)
&=
\bm{p}(0),
\label{eq:p^*(0)}
\\
\frac{\partial \bm{p}^*(z,t)}{\partial t}
&=
\bm{p}^*(z,t)\bigl[\bm{Q}(t)-\lambda(t)(1-z)\bm{B}\bigr],
\label{eq:p^*-differential equation}
\end{align}
where $\bm{p}(0)$ is given by \eqref{eq:p(0)}, and 
$\bm{B}$ denotes an $|\Omega| \times |\Omega|$
diagonal matrix representing state transitions due to customer balking:
\[
[\bm{B}]_{(k,\ell,s),(k',\ell',s')}
=
\left\{
\begin{aligned}
&1-\beta_\ell,&& (k',\ell',s') = (k,\ell,s),\\
&0,&& \text{otherwise}.
\end{aligned}
\right.
\]
\end{lemma}
\begin{proof}
Note that \eqref{eq:p^*(0)} is obvious from $A^\balk(0) = 0$.
By definition, as $\varDelta t \to 0+$,
\[
\bm{p}^*(z,t+\varDelta t)
=
\bm{p}^*(z,t)\{\bm{I}+(\bm{Q}(t) - \lambda(t)\bm{B})\varDelta t\}
+
\bm{p}^*(z,t)
\cdot
z \lambda(t)\bm{B}\varDelta t 
+
\bm{o}(\varDelta t),
\]
which implies (\ref{eq:p^*-differential equation}).
\end{proof}

We then characterize the conditional PGF $\pi_{k,\ell,s}^*(z,t \mid K)$
(defined as \eqref{eq:pi^*-def}) in terms of the unconditional PGF.
Recall that the conditional and unconditional state probabilities 
$\pi_{k,\ell,s,b}(t \mid K)$ and $p_{k,\ell,s,b}(t)$ are related as
(\ref{eq:pi-straightforward-approach}).
As explained in the previous section, the key quantity in developing
an efficient computational method is the conditional probability 
$\Pr[A^\join(T)=K \mid A^\join(t)=k, L(t)=\ell, S(t)=s]$.

To facilitate efficient analysis of this quantity, we introduce a
conditional probability $q_{k,\ell,s}(u,K)$, parametrized backward in
time:
\begin{align}
q_{k,\ell,s}(u,K)
\ceq
\Pr[A^\join(T)=K
\mid A^\join(T-u)=k, L(T-u)=\ell, S(T-u)=s],
\ \notag\\
(k,\ell,s) \in \Omega,\ 0 \leq u \leq T.
\label{eq:q-def}
\end{align}
By definition, $q_{k,\ell,s}(u,K)$ represents 
the conditional probability that exactly $K$ customers
have joined before time $T$, given the system state $(k,\ell,s)$ at time $T-u$.
Let $\bm{q}(u,K)$ denote a $|\Omega| \times 1$ vector consisting of the
elements $q_{k,\ell,s}(u,K)$, arranged in lexicographic order. 
\begin{lemma}\label{lem:q}
$\bm{q}(u,K)$ is determined by the following linear equation 
\begin{align}
q_{k,\ell,s}(0,K)
&=
\left\{
\begin{aligned}
&1, && k=K,\ \ell=0,1,\ldots,K,\ s\in\mathcal{S},\\
&0, && \text{otherwise},
\end{aligned}
\right.
\label{eq:q(0,K)}
\\
\frac{\partial \bm{q}(u,K)}{\partial u}
&=
\bm{Q}(T-u)\bm{q}(u,K).
\label{eq:q-differential equation}
\end{align}
\end{lemma}
\begin{proof}
As the boundary condition \eqref{eq:q(0,K)} follows 
immediately from the definition of $\bm{q}(u,K)$, we prove 
\eqref{eq:q-differential equation} below.
From the Markov property of the system state $(A^\join(t),
L(t), S(t)) _{t \geq 0}$, we have for $\varDelta t > 0$,
\begin{align*}
\lefteqn{%
q_{k,\ell,s}(u+\varDelta t)
}\;\;
\\
&=
\Pr[A^\join(T)=K
\\ &\qquad \qquad
\mid A^\join(T-u-\varDelta t)=k, L(T-u-\varDelta t)=\ell, S(T-u-\varDelta t)=s]
\\
&=
\sum_{(k',\ell',s')\in\Omega}
\Bigl(
\Pr[A^\join(T-u)=k', L(T-u)=\ell', S(T-u)=s' 
\\[-1.5ex]
&\hspace{6em}
\mid A^\join(T-u-\varDelta t)=k, L(T-u-\varDelta t)=\ell, S(T-u-\varDelta t)=s] \\
&\hspace{4em} \cdot
\Pr[A^\join(T)=K
\mid A^\join(T-u)=k', L(T-u)=\ell', S(T-u)=s']
\Bigr).
\end{align*}
Therefore, using the elements $[\bm{Q}(t)]_{(k,\ell,s),(k',\ell',s')}$ of
the transition matrix $\bm{Q}(t)$, we have 
\begin{align*}
q_{k,\ell,s}(u+\varDelta t)
&=
(1+[\bm{Q}(T-u-\varDelta t)]_{(k,\ell,s),(k,\ell,s)} \varDelta t)
q_{k,\ell,s}(u,K)\\
&\qquad 
+
\sum_{(k',\ell',s') \in \Omega\backslash(k,\ell,s)}
[\bm{Q}(T-u-\varDelta t)]_{(k,\ell,s),(k',\ell',s')} \varDelta t
q_{k',\ell',s'}(u,K)\\
&\qquad
+
o(\varDelta t),
\end{align*}
as $\varDelta t \to 0+$. 
This equation is further rewritten in matrix form as 
\[
\bm{q}(u+\varDelta t)
=
(\bm{I}+\bm{Q}(T-u-\varDelta t)\varDelta t)
\bm{q}(u,K)
+
\bm{o}(\varDelta t),
\]
which implies \eqref{eq:q-differential equation}. 
\end{proof}

The conditional PGF $\pi_{k,\ell,s}^*(z,t \mid K)$ is thus obtained in
terms of $p_{k,\ell,s}^*(t)$ and $q_{k,\ell,s}(u,K)$:
\begin{theorem}\label{thm:pi^*}
The PGF $\pi_{k,\ell,s}^*(z,t \mid K)$ conditional on exactly $K$ customers
joined in $[0,T]$ is given by 
\begin{equation}
\pi_{k,\ell,s}^*(z,t \mid K)
=
\frac{p_{k,\ell,s}^*(z,t) q_{k,\ell,s}(\red{T-t},K)}
{\Pr[A^\join(T)=K]},
\quad 
(k,\ell,s) \in \Omega,\ 0 \leq t \leq T,
\label{eq:thm:pi^*}
\end{equation}
where $p_{k,\ell,s}^*(\red{z},t)$ and $q_{k,\ell,s}(u,K)$ are given by Lemma
\ref{lem:p^*} and Lemma \ref{lem:q}.
\end{theorem}
\begin{proof}
Theorem \ref{thm:pi^*} follows immediately from (\ref{eq:pi-straightforward-approach}).
\end{proof}

Theorem \ref{thm:pi^*} provides the complete characterization of the state
probabilities of the system, conditional on the event that exactly $K$
customers have joined in $[0,T]$.
As $\pi_{k,\ell,s}^*(z,t \mid K)$ represents the number of balked customers
in the form of PGF, we can derive its moments accordingly.

Letting $\bm{p}^{(i)}(t)$ denote a $1 \times |\Omega|$ vector
consisting of the elements $p_{k,\ell,s}^{(i)}(t)$ 
(cf. \eqref{eq:p^(0)-def} and \eqref{eq:p^(i)-def}) arranged in
lexicographic order, we obtain the following results from Lemma
\ref{lem:p^*} and Theorem \ref{thm:pi^*}:
\begin{lemma}\label{lem:p^(i)}
$\bm{p}^{(i)}(t)$ ($i=0,1,\ldots$) is determined by the following
differential equation:
\begin{align}
\bm{p}^{(i)}(0) &= \bm{0},
\quad
i=1,2,\ldots,
\label{eq:p^(i)(0)}
\\
\frac{\dd \bm{p}^{(i)}(t)}{\dd t}
&=
\bm{p}^{(i)}(t)\bm{Q}(t)
+
i \bm{p}^{(i-1)}(t)\lambda(t)\bm{B},
\quad
0 \leq t \leq T,\ i=1,2,\ldots.
\label{eq:p^(i)-differential equation}
\end{align}
\end{lemma}
\begin{proof}
Differentiating both sides of \eqref{eq:p^*(0)} 
and \eqref{eq:p^*-differential equation} 
with respect to $z$ up to the $i$-th order 
and substituting $z = 1$,
we obtain \eqref{eq:p^(i)(0)} and \eqref{eq:p^(i)-differential
equation}.
\end{proof}
\begin{theorem}\label{thm:pi^(i)}
$\pi_{k,\ell,s}^{(i)}(t \mid K)$ are given by
\begin{equation}
\pi_{k,\ell,s}^{(i)}(t \mid K)
=
\frac{p_{k,\ell,s}^{(i)}(t) q_{k,\ell,s}(T-t,K)}
{\Pr[A^\join(T)=K]},
\quad 
0 \leq t \leq T,\ i = 0,1,\ldots,
\label{eq:thm:pi^(i)}
\end{equation}
where $q_{k,\ell,s}(u,K)$ and $p_{k,\ell,s}^{(i)}(t)$ are
given by Lemma \ref{lem:q} and Lemma \ref{lem:p^(i)}.
\end{theorem}
\begin{proof}
\eqref{eq:thm:pi^(i)} follows immediately from \eqref{eq:thm:pi^*} and
the definitions of $\pi_{k,\ell,s}^{(i)}(t \mid K)$ and $p_{k,\ell,s}^{(i)}(t)$.
\end{proof}

The results above show that for given system parameters $\lambda(t)$,
$\beta_{\ell}$, and $\bm{Q}(t)$, we can compute the conditional PGF
$\pi_{k, \ell,s}^*(z,t \mid K)$ and the moments $\pi_{k,\ell,s}^{(i)}(t \mid K)$ 
by solving the linear differential equations in
Lemma \ref{lem:p^*}, Lemma \ref{lem:q}, and Lemma \ref{lem:p^(i)}.
In the following sections, we develop a detailed computational algorithm
focusing on the case that $\lambda(t)$ and $\bm{Q}(t)$ are
piecewise constant in time $t$ and we present numerical examples, 
which demonstrate the feasibility of computation based on these
general results.

\section{Special case: Piecewise time-homogeneous systems}
\label{sec:special case}

In this section, we develop a detailed computational algorithm for the case
with piecewise constant $\lambda(t)$ and $\bm{Q}(t)$.
More specifically, the time interval $(0,T]$ is divided into $N$
($N\in\mathbb{N}$) subintervals $(T_0,T_1], (T_1,T_2], \ldots,
(T_{N-1},T_N]$ with $T_0=0$ and $T_N=T$, and the arrival rate
$\lambda(t)$ and transition rate matrix $\bm{Q}(t)$ are assumed
constant during each interval: 
\begin{align}
\lambda(t)=\lambda_n,
\quad
\bm{Q}(t)=\bm{Q}_n,
\qquad 
T_{n-1} < t \leq T_n,\ n=1,2,\ldots,N.
\label{eq:piecewise-constant}
\end{align}

In this case, the unconditional PGF $\bm{p}^*(z,t)$ characterized as Lemma \ref{lem:p^*} 
is given recursively by
\begin{align}
\bm{p}^*(z,t)
=
\bm{p}^*(z,T_{n-1}) \exp((\bm{Q}_n-\lambda_n(1-z)\bm{B})(t-T_{n-1})),
\quad \notag\\
T_{n-1} < t \leq T_n,\ n=1,2,\ldots,N,
\label{eq:p^*-special case}
\end{align}
with $\bm{p}^*(z,0)$ given by \eqref{eq:p^*(0)}.
Similarly, the conditional probability $\bm{q}(u,K)$ 
characterized as Lemma \ref{lem:q} is given recursively by
\begin{align}
\bm{q}(u,K)
=
\exp(\bm{Q}_{N-n+1} \cdot (u-T+T_{N-n+1}))
\bm{q}(T-T_{N-n+1},K),
\qquad \notag \\
T-T_{N-n+1} < u \leq T-T_{N-n},\ n=1,2,\ldots,N,
\label{eq:q-special case}
\end{align}
with $\bm{q}(0,K)$ given by \eqref{eq:q(0,K)}.
Note here that the equality at right endpoints
$t = T_n$ in \eqref{eq:p^*-special case} and 
$u = T-T_{N-n}$ in \eqref{eq:q-special case} follows
from the continuity of $\bm{p}^*(z,t)$ and $\bm{q}(u,K)$. 
Therefore, we obtain the time-dependent PGF $\pi_{k,\ell,s}^*(z,t \mid K)$ 
from Theorem \ref{thm:pi^*}.

For the factorial moments $\pi_{k,\ell,s}^{(i)}(t \mid K)$, the
following representation of unconditional factorial moments 
$\bm{p}^{(i)}(t)$ leads to efficient computation:
\begin{equation}
\widetilde{\bm{p}}^{(J)}(t)
\ceq
\left[
\begin{array}{cccc}
\ds \frac{\bm{p}^{(0)}(t)}{0!} & \ds \frac{\bm{p}^{(1)}(t)}{1!} & \cdots & \ds \frac{\bm{p}^{(J)}(t)}{J!}
\end{array}
\right],
\quad
J = 0,1,\ldots.
\label{eq:tildep-def}
\end{equation}
By definition, $\widetilde{\bm{p}}^{(J)}(t)$ 
denotes a vector of time-dependent joint probability 
$\bm{p}(t) = \bm{p}^{(0)}(t)$ and the factorial moments
$\bm{p}^{(i)}(t)$ ($i=1,2,\ldots,J$), weighted according 
to their order. Under the assumption \eqref{eq:piecewise-constant}, $\widetilde{\bm{p}}^{(J)}(t)$ is obtained as follows:
\begin{theorem}\label{thm:tildep}
For a fixed $J \in \{1,2,\ldots\}$, $\widetilde{\bm{p}}^{(J)}(t)$ is given by
\begin{align}
\widetilde{\bm{p}}^{(J)}(0)
&=
\underbrace{
\left[
\begin{array}{cccc}
\bm{p}(0) & \bm{0} & \cdots & \bm{0}
\end{array}
\right]}_{\ds J+1},  
\label{eq:tildep(0)}
\\
\widetilde{\bm{p}}^{(J)}(t)
&=
\widetilde{\bm{p}}^{(J)}(T_{n-1})
\exp(\widetilde{\bm{Q}}_n^{(J)}(t-T_{n-1})),
\quad
T_{n-1} < t \leq T_n,\ n=1,2,\ldots,N,
\label{eq:thm:tildep}
\end{align}
where $\widetilde{\bm{Q}}_n^{(J)}$ is defined as
\begin{align}
\widetilde{\bm{Q}}_n^{(J)}
&=
\underbrace{
\left[
\begin{array}{ccccc}
\bm{Q}_n & \lambda_n\bm{B} & \bm{O} & \cdots & \bm{O}\\
\bm{O} & \bm{Q}_n & \lambda_n\bm{B} & \cdots & \bm{O}\\
\bm{O} & \bm{O} & \bm{Q}_n & \cdots & \bm{O}\\
\vdots & \vdots & \vdots & \ddots & \vdots\\
\bm{O} & \bm{O} & \bm{O} & \cdots & \bm{Q}_n
\end{array}
\right]
}_{\ds (J+1)\times(J+1)}.
\label{eq:tildeQ-def}
\end{align}
\end{theorem}
\begin{proof}
As \eqref{eq:tildep(0)} follows immediately 
from \eqref{eq:p^(i)(0)}, we prove \eqref{eq:thm:tildep} below.
Using \eqref{eq:piecewise-constant}, we 
rewrite \eqref{eq:p-differential equation} and 
\eqref{eq:p^(i)-differential equation} as 
\begin{alignat*}{2}
\frac{\dd \bm{p}^{(0)}(t)}{\dd t}
&=
\bm{p}^{(0)}(t)\bm{Q}_n,
\quad
i = 0,\ 
T_{n-1} < t \leq T_n,\ n = 1,2,\ldots,N,
\\
\frac{\dd \bm{p}^{(i)}(t)}{\dd t}
&=
\bm{p}^{(i)}(t)\bm{Q}_n
+
i \bm{p}^{(i-1)}(t) \lambda_n \bm{B},\\
&\hspace{6.8em}
i = 1,2,\ldots,\ 
T_{n-1} < t \leq T_n,\ n = 1,2,\ldots,N,
\end{alignat*}
which implies
\begin{alignat*}{2}
\frac{\dd}{\dd t} 
\left[
\frac{\bm{p}^{(0)}(t)}{0!}
\right]
&=
\frac{\bm{p}^{(0)}(t)}{0!} \cdot \bm{Q}_n,
\quad
i=0,\ T_{n-1} < t \leq T_n,\ n = 1,2,\ldots,N,
\\
\frac{\dd}{\dd t} 
\left[
\frac{\bm{p}^{(i)}(t)}{i!}
\right]
&=
\frac{\bm{p}^{(i)}(t)}{i!} \cdot \bm{Q}_n
+
\frac{\bm{p}^{(i-1)}(t)}{(i-1)!} \cdot \lambda_n \bm{B},\\
&\hspace{7.8em}
i = 1,2,\ldots,\ 
T_{n-1} < t \leq T_n,\ n = 1,2,\ldots,N.  
\end{alignat*}
It then follows from \eqref{eq:tildep-def} and
\eqref{eq:tildeQ-def} that
\[
\frac{\dd \widetilde{\bm{p}}^{(J)}(t)}{\dd t}
=
\widetilde{\bm{p}}^{(J)}(t) \widetilde{\bm{Q}}_n^{(J)},
\quad
T_{n-1} < t \leq T_n,\ n = 1,2,\ldots,N.
\]
Therefore, we obtain \eqref{eq:thm:tildep} from this differential equation.
\end{proof}
Theorem \ref{thm:tildep} shows that the unconditional factorial
moments $\bm{p}^{(i)}(t)$ ($i=0,1,\ldots,J$) up to the $J$-th order
can be obtained simultaneously by computing
$\widetilde{\bm{p}}^{(J)}(t)$.
More specifically, we have
\begin{alignat}{2}
\bm{p}(t)
&=
\widetilde{\bm{p}}^{(J)}(t) \widetilde{\bm{I}}_0^{(J)},
\label{eq:p}
\\
\bm{p}^{(i)}(t)
&=
\widetilde{\bm{p}}^{(J)}(t) \widetilde{\bm{I}}_i^{(J)} i!,
&&\quad
i=1,2,\ldots,J,
\label{eq:p^(i)}
\end{alignat}
where $\widetilde{\bm{I}}_i^{(J)}$ is given by
\[
\widetilde{\bm{I}}_i^{(J)}
\ceq
\Bigl[
\underbrace{%
\begin{array}{ccc}
\bm{O} & \cdots & \bm{O}
\end{array}}_\red{\ds i}
\begin{array}{c}
\bm{I}
\end{array}
\underbrace{%
\begin{array}{ccc}
\bm{O} & \cdots & \bm{O}
\end{array}}_\red{\ds J-i}
\Bigr],
\quad
i=0,1,\ldots,J,\ J \in \mathbbm{N},
\]
i.e., it has an identity matrix $\bm{I}$ at the 
$i$-th block, and zero matrices $\bm{O}$ elsewhere.
Therefore, utilizing this representation, we can compute the
conditional factorial moments $\pi_{k,\ell,s}^{(i)}(t \mid K)$ efficiently
from Theorem~\ref{thm:pi^(i)}. 

To develop a stable computational algorithm, it is essential to use
the uniformization technique \cite[pp.~154--156]{Tijms1994} for
computing the matrix exponentials.
More specifically, we rewrite \eqref{eq:q-special case} as
\begin{align}
\bm{q}(u,K)
=
\sum_{m=0}^{\infty} \Poi(\theta_\red{N-n+1}(u-T+T_{N-n+1}),m)
(\bm{P}_n)^m \bm{q}(T-T_{N-n+1}, K),
\qquad \notag\\
T-T_{N-n+1} \leq u \leq T-T_{N-n},\ n=1,2,\ldots,N,
\label{eq:q-uniformization}
\end{align}
where $\theta_n$ and $\bm{P}_n$ are defined as
\begin{align}
\theta_n 
&\ceq 
\max_{i}|[\bm{Q}_n]_{i,i}| + \lambda_n,
\label{eq:theta-def}
\\
\bm{P}_n
&\ceq
\bm{I} + \theta_n^{-1}\bm{Q}_n,
% \label{eq:P-def}
\notag
\end{align}
and $\Poi(a,m)$ denotes the probability 
mass function of the Poisson distribution with mean $a$:
\[
\Poi(a,m) \ceq \frac{a^m}{m!}e^{-a}.
\]

Note that the right-hand side of \eqref{eq:q-uniformization} involves
only additions and multiplications of non-negative numbers, which
ensures numerical stability by avoiding the loss of significant
digits.  
Although the choice of $\theta_n$ is arbitrary as long as it satisfies
$\theta_n \geq \max_{i} |[\bm{Q}_n]_{i,i}|$, we adopt the specific
form used here for convenience in computing
$\widetilde{\bm{p}}^{(J)}(t)$, as discussed below.

For the computation of $\widetilde{\bm{p}}^{(J)}(t)$ using \eqref{eq:thm:tildep},
it is important to note that $\widetilde{\bm{Q}}_n^{(J)}$ (defined as
\eqref{eq:tildeQ-def}) is not necessarily a proper transition rate
matrix as its row sums may exceed zero.
To improve numerical stability, we factor out a scalar exponential term:
\[
\exp(\widetilde{\bm{Q}}_n^{(J)} \cdot (t-T_{n-1}))
=
e^{\lambda_n(t-T_{n-1})}
\exp((\widetilde{\bm{Q}}_n^{(J)}-\lambda_n\bm{I})(t-T_{n-1})),
\]
where $\widetilde{\bm{Q}}_n^{(J)}-\lambda_n\bm{I}$ represents a proper
transition rate matrix.
Using this relation, we rewrite \eqref{eq:thm:tildep} as
\begin{align}
\widetilde{\bm{p}}^{(J)}(t)
=
e^{\lambda_n(t-T_{n-1})}
\sum_{m=0}^{\infty}
\Poi(\theta_n(t-T_{n-1}), m)
\widetilde{\bm{p}}^{(J)}(T_{n-1})
(\widetilde{\bm{P}}_n^{(J)} - \theta_n^{-1}\lambda_n\bm{I})^m,
\ \notag\\
T_{n-1} < t \leq T_n,\ n=1,2,\ldots,N,
\label{eq:tildep-uniformization}
\end{align} 
where 
\begin{align}
\widetilde{\bm{P}}_n^{(J)}
&\ceq
\bm{I} + \theta_n^{-1}\widetilde{\bm{Q}}_n^{\red{(J)}}
=
\underbrace{
\left[
\begin{array}{ccccc}
\bm{P}_n & \theta_n^{-1}\lambda_n\bm{B} & \bm{O} & \cdots & \bm{O}\\
\bm{O} & \bm{P}_n & \theta_n^{-1}\lambda_n\bm{B} & \cdots & \bm{O}\\
\bm{O} & \bm{O} & \bm{P}_n & \cdots & \bm{O}\\
\vdots & \vdots & \vdots & \ddots & \vdots\\
\bm{O} & \bm{O} & \bm{O} & \cdots & \bm{P}_n
\end{array}
\right]}_{\ds (J+1)\times(J+1)}. \notag
\end{align}
We can verify that the right-hand side of \eqref{eq:tildep-uniformization}
involves only additions and multiplications of non-negative numbers.

Furthermore, to compute the term 
$\widetilde{\bm{p}}^{(J)}(T_{n-1}) (\widetilde{\bm{P}}_n^{(J)} - \theta_n^{-1}\lambda_n \bm{I})^m$  
in \eqref{eq:tildep-uniformization} efficiently,  
we leverage the block structure and sparsity of $\widetilde{\bm{P}}_n^{(J)}$.  
To this end, we define $\bm{x}_n(m)$ as 
\begin{align*}
\bm{x}_n(m)
\ceq 
\widetilde{\bm{p}}^{(J)}(T_{n-1})
(\widetilde{\bm{P}}_n^{(J)}-\theta_n^{-1}\lambda_n\bm{I})^m
=
[\bm{x}_n^{(0)}(m), \bm{x}_n^{(1)}(m), \ldots, \bm{x}_n^{(J)}(m)],
\quad\\
m=0,1,\ldots,
\end{align*}
which satisfies the recurrence
\[
\bm{x}_n(0) = \widetilde{\bm{p}}^{(J)}(T_{n-1}),
\quad
\bm{x}_n(m) 
= 
\bm{x}_n(m-1)(\widetilde{\bm{P}}_n^{(J)}-\theta_n^{-1}\lambda_n\bm{I}),
\quad
m=1,2,\ldots.
\]
Each block $\bm{x}_n^{(i)}(m)$ of $\bm{x}_n(m)$ can then be 
computed recursively as
\begin{alignat}{2}
\bm{x}_n^{(0)}(m+1)
&=
\bm{x}_n^{(0)}(m)(\bm{P}_n-\theta_\red{n}^{-1}\lambda_n\bm{I}),
&&\quad
i=0,
\label{eq:x^0(m)}
\\
\bm{x}_n^{(i)}(m+1)
&=
\bm{x}_n^{(i-1)}(m)\theta^{-1}\lambda_n\bm{B}
+
\bm{x}_n^{(i)}(m)(\bm{P}_n-\theta_\red{n}^{-1}\lambda_n\bm{I}),
&&\quad
i=1,2,\ldots.
\label{eq:x^i(m)}
\end{alignat}

\begin{remark}
For the uniformization technique to work properly,
both $\bm{P}_n$ in \eqref{eq:q-uniformization} and  
$\widetilde{\bm{P}}_n^{(J)} - \theta_n^{-1} \lambda_n \bm{I}$ in  
\eqref{eq:tildep-uniformization} must be substochastic matrices.  
The former condition is satisfied when  
$\theta_n \geq \max_i |[\bm{Q}_n]_{i,i}|$,  
while the latter requires  
$\theta_n \geq \max_i |[\bm{Q}_n]_{i,i}| + \lambda_n$.  
Therefore, we adopt the specific form of $\theta_n$ given in
\eqref{eq:theta-def} to ensure that both conditions hold.
\end{remark}

Figure~\ref{fig:computational procedure} summarizes the computational
procedure for evaluating the factorial moments
$\pi_{k,\ell,s}^{(i)}(t \mid K)$. In this procedure, the truncation point
$m = M$ for the infinite sums in \eqref{eq:q-uniformization} and
\eqref{eq:tildep-uniformization} is treated as an input parameter.

\begin{figure}[!t]
\centering
\noindent
\begin{small}
\fbox{\hspace*{3mm}
\begin{minipage}{130mm}
\mbox{}
\\
\textbf{Input}: 
$T$, $K$, $N$, $T_n$ ($n=1,2,\ldots,N$), 
$\lambda_n$ ($n=1,2,\ldots,N$), $\bm{B}$, \\
\hspace*{3.3em}
$\bm{Q}_n$ ($n=1,2,\ldots,N$), $\bm{p}(0)$, 
$t$, $J$, $N^*(t)$, and $M$.
\\
\textbf{Output}: 
$\pi_{k,\ell,s}^{(i)}(t \mid K)$ 
($(k,\ell,s) \in \Omega$, $i=0,1,\ldots,J$).
\\[3mm]
\textbf{Step 1}: 
Computation of $\widetilde{\bm{p}}^{(J)}(T_{N^*(t)-1})$ 
and $\widetilde{\bm{p}}^{(J)}(T)$.
\\
\mbox{}\quad Set $\widetilde{\bm{p}}^{(J)}(0)$ as in \eqref{eq:tildep(0)}.
\\
\mbox{}\quad \textbf{for} $n=1$ to $N$ \textbf{do}
\\
\mbox{}\qquad Compute $\widetilde{\bm{p}}^{(J)}(T_n)$ 
using \eqref{eq:tildep-uniformization} and the truncation point $m=M$.
\\
\mbox{}\quad \textbf{endfor}
\\[3mm]
\textbf{Step 2}:
Computation of $\bm{q}(T-T_{N^*(t)},K)$.
\\
\mbox{}\quad Set $\bm{q}(0,K)$ as in \eqref{eq:q(0,K)}.
\\
\mbox{}\quad \textbf{if} $N^*(t) < N$ \textbf{then}
\\
\mbox{}\qquad \textbf{for} $n=1$ to $N-N^*(t)$ \textbf{do}
\\
\mbox{}\qquad\quad Compute $\bm{q}(T-T_{N-n},K)$ using
\eqref{eq:q-uniformization} and the truncation point $m=M$.
\\
\mbox{}\qquad \textbf{endfor}
\\
\mbox{}\quad \textbf{endif}
\\[3mm]
\textbf{Step 3}:
Computation of output $\pi_{k,\ell,s}^{(i)}(t \mid K)$.
\\
\mbox{}\quad Compute $\widetilde{\bm{p}}^{(J)}(t)$ 
using \eqref{eq:tildep-uniformization}
and $\bm{q}(T-t,K)$ using \eqref{eq:q-uniformization}, with the 
truncation point $m=M$.
\\
\mbox{}\quad \textbf{for} $i=0$ to $J$ \textbf{do}
\\
\mbox{}\qquad Compute $\bm{p}^{(i)}(t)$ by \eqref{eq:p} or \eqref{eq:p^(i)}.
\\
\mbox{}\qquad \textbf{for} $(k,\ell,s) \in \Omega$ \textbf{do}
\\
\mbox{}\qquad\quad Compute $\pi_{k,\ell,s}^{(i)}(t \mid K)$ 
by \eqref{eq:thm:pi^(i)}.
\\
\mbox{}\qquad \textbf{endfor}
\\
\mbox{}\quad \textbf{endfor}
\\[-2mm]
\mbox{}
\end{minipage}
\hspace*{3mm}}
\end{small}
\caption{Computational procedure for $\pi_{k,\ell,s}^{(i)}(t \mid K)$ 
($t \in (0,T]$). $N^*(t)$ denotes a non-negative integer that 
satisfies $T_{N^*(t)-1} < t \leq T_{N^*(t)}$.
}
\label{fig:computational procedure}
\end{figure}

\section{Numerical examples}

In this section, we present numerical examples for an M/PH/1 queue  
and investigate how condition that the total number of joined customers is given
affects the time-dependent behavior of the number of customers in the system.
This setting corresponds to the case considered in Section~4 with $N = 1$,  
and we fix the arrival rate as $\lambda_1 = \lambda$ throughout this section.
Service times are assumed to be i.i.d.\ according to a phase-type distribution. 
More specifically, we use a subclass of phase-type 
distributions \cite[pp.~358--359]{Tijms1994} characterized by the mean
and coefficient of variation $C_V$:
(i) a mixture of Erlang distributions with
$k$ and $k+1$ stages ($0 < C_V < 1$), (ii) an exponential distribution
($C_V = 1$), and (iii) a balanced hyper-exponential distribution ($C_V > 1$).
Throughout the experiments, the mean service time is fixed to one.
Unless otherwise mentioned, we set the observation horizon to $T = 100$  
and the total number of joined customers to $K = 100$.
We provide a detailed computational procedure for this model in Appendix~\ref{sec:M/PH/1}.

We begin by examining the time-dependent distribution of the  
number of customers in the system, conditional on the total number of joined customers.  
We set the entry probability $\beta_\ell$ as
\begin{equation}
\beta_\ell
=
\max\left(1-\frac{\ell}{50}, 0\right),
\quad \ell = 0,1,\ldots,100.
\label{eq:beta}
\end{equation}
Figure~\ref{fig:heatmap} shows heatmaps of the distribution
of the number of customers in the system for $C_V = 1.0$ and different
values of $\lambda$,
conditional on exactly $K = 100$ joined customers.
Each panel also includes $10$ sample paths obtained from simulation.  
These sample paths tend to pass through regions of higher density in
the heatmaps, highlighting the consistency between the simulated
trajectories and the computed distributions.

These heatmaps and sample paths reveal behaviors
that may initially seem counterintuitive.
For example, in Figure~\ref{fig:heatmap} (d), the number of customers
follows a unimodal trajectory (first increasing and then
decreasing) despite the constant arrival rate $\lambda(t) = \lambda = 2.0$.  
This illustrates how condition that the total number of joined customers
($K = 100$) is given can induce dynamics that deviate fundamentally from those
in conventional queueing models, where the system converges
to the steady state over time.

\begin{figure}[!t]
\centering
\subfigure[$\lambda=0.5$]{
\includegraphics{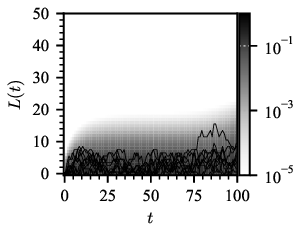}}
\qquad
\subfigure[$\lambda=1.1$]{
\includegraphics{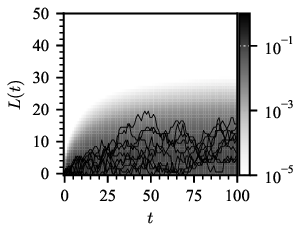}}
\\
\subfigure[$\lambda=1.5$]{
\includegraphics{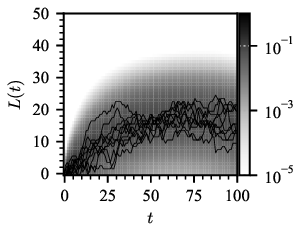}}
\qquad 
\subfigure[$\lambda=2.0$]{
\includegraphics{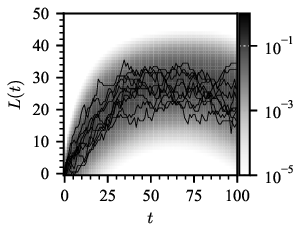}}
\caption{The distribution of the number of 
customers in the system plotted with simulated sample paths 
($C_V=1.0$ and $\beta_\ell$ is given by \eqref{eq:beta}).}
\label{fig:heatmap}
\end{figure}

To investigate the mechanism behind this behavior,  
Figure~\ref{fig:lambda} presents the time evolution of several key system metrics,  
all conditional on exactly $K$ customers having joined by time $T$:
\begin{itemize}
\item the expected number of customers in the system $\E[L(t) \mid A^\join(T)=K]$,

\item the increment in the expected cumulative number of joined customers $\E[\varDelta A^\join(t) \mid A^\join(T)=K]$,

\item the increment in the expected cumulative number of balked customers $\E[\varDelta A^\balk(t) \mid A^\join(T)=K]$,

\item the increment in the expected cumulative number of arrivals $\E[\varDelta A(t) \mid A^\join(T)=K]$,

\item the increment in the expected cumulative number of departures $\E[\varDelta D(t) \mid A^\join(T)=K]$, and

\item the system utilization $\Pr[L(t) \geq 1 \mid A^\join(T)=K]$,
\end{itemize}
where for any time-dependent quantity $X(t)$, we define the increment
$\varDelta X(t) \ceq X(t + 1) - X(t)$.
From Figure~\ref{fig:lambda} (a), we observe that
for $\lambda = 1.3$, $1.5$, and $2.0$, the number of customers in the system  
first increases and then decreases over time.
In contrast, for $\lambda = 0.5$ and $0.8$,
the number of customers remains nearly flat for most of the observation period,  
with slight increases near the beginning and end.

By definition, the change in the number of customers in the system
satisfies
$\varDelta L(t) = \varDelta A^{\join}(t) - \varDelta D(t)$,
so the transient behavior of $L(t)$ can be explained
in terms of the numbers of joined and departed customers.
Figure~\ref{fig:lambda} (b) shows that the mean joining rate
$\E[\varDelta A^{\join}(t) \mid A^{\join}(T) = K]$
exhibits a bias under the conditioning $A^{\join}(T) = K$.
In particular, larger values of the arrival rate $\lambda$
lead to more joining events in the early stage.
This bias arises because, when $\lambda \gg K/T = 1$,
a larger number of balking events is necessary to limit
the total number of joined customers to $K = 100$.
To satisfy this constraint, the system tends to admit more customers
early on, causing $L(t)$ to increase rapidly and thereby
creating more opportunities for balking.

Figure~\ref{fig:lambda} (a), (c), and (d) further show
that the bias in the conditional joining rate
$\E[\varDelta A^{\join}(t) \mid A^{\join}(T) = K]$
originates mainly from a bias in the conditional arrival rate
$\E[\varDelta A(t) \mid A^{\join}(T) = K]$,
rather than from the balking behavior itself.
In fact, the conditional balking rate
$\E[\varDelta A^{\balk}(t) \mid A^{\join}(T) = K]$ 
closely follows the system congestion 
level $\E[L(t) \mid A^{\join}(T) = K]$.
Therefore, we conclude that the counterintuitive time-dependent  
behavior of $\E[L(t) \mid A^{\join}(T) = K]$  
mainly stems from a bias introduced in the arrival process 
due to the condition that the total number of joined customers is given.

\fboxsep=0pt
\fboxrule=0.7pt
\begin{figure}[!t]
\centering
\subfigure{%
\fbox{
\includegraphics{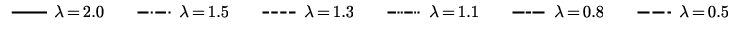}}}
\setcounter{subfigure}{0}
\subfigure[\scriptsize \mbox{$\E[L(t) \mid A^\join(T)=K]$}]{%
\includegraphics{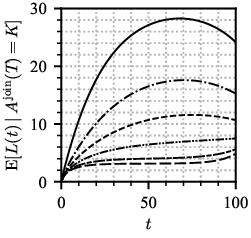}}
\subfigure[\scriptsize \mbox{$\E[\varDelta A^\join(t) \mid A^\join(T)=K]$}]{%
\includegraphics{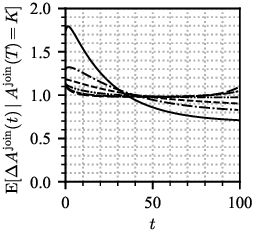}}
\subfigure[\scriptsize \mbox{$\E[\varDelta A^\balk(t) \mid A^\join(T)=K]$}]{%
\includegraphics{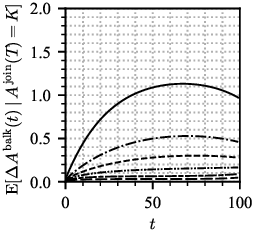}}
\\
\subfigure[\scriptsize \mbox{$\E[\varDelta A(t) \mid A^\join(T)=K]$}]{%
\includegraphics{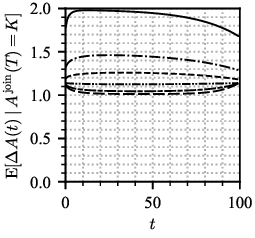}}
\subfigure[\scriptsize \mbox{$\E[\varDelta D(t) \mid A^\join(T)=K]$}]{%
\includegraphics{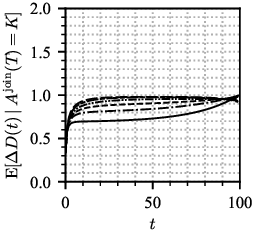}}
\subfigure[\scriptsize \mbox{$\Pr[L(t) \geq 1 \mid A^\join(T)=K]$}]{%
\includegraphics{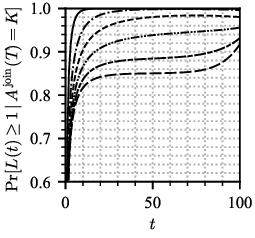}}
\caption{%
Transient behavior of several performance metrics 
for different values of the arrival rate $\lambda$ 
($C_V=1.0$ and $\beta_\ell$ is given by \eqref{eq:beta}).}
\label{fig:lambda}
\end{figure}

Figure~\ref{fig:A^join(T)} complements this observation by showing the
unconditional mean number of joined customers $\E[A^\join(T)]$
in $[0,T]$ as a function of $\lambda$.  
By examining Figure~\ref{fig:lambda} (a) alongside 
Figure~\ref{fig:A^join(T)}, we see how $\E[A^\join(T)]$  
is related to the time-dependent behavior of $\E[L(t) \mid A^{\join}(T)
= K]$: for $\E[A^\join(T)] > K$, the trajectory typically peaks in
the middle of the interval, while for $\E[A^\join(T)] < K$, it tends to
increase steadily over time.

\begin{figure}[!t]
\centering
\includegraphics{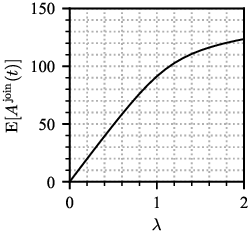}
\caption{$\E[A^\join(T)]$ with respect to $\lambda$, 
where $C_V=1.0$ and $\beta_\ell$ is given by \eqref{eq:beta}.}
\label{fig:A^join(T)}
\end{figure}

\begin{figure}[!t]
\centering
\subfigure{%
\fbox{
\includegraphics{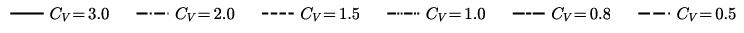}}}
\\
\setcounter{subfigure}{0}
\subfigure[\scriptsize \mbox{$\E[L(t) \mid A^\join(T)=K]$}]{%
\includegraphics{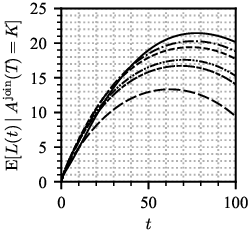}}
\subfigure[\scriptsize \mbox{$\E[\varDelta A^\join(t) \mid A^\join(T)=K]$}]{%
\includegraphics{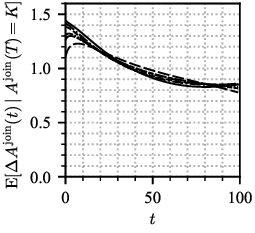}}
\subfigure[\scriptsize \mbox{$\E[\varDelta A^\balk(t) \mid A^\join(T)=K]$}]{%
\includegraphics{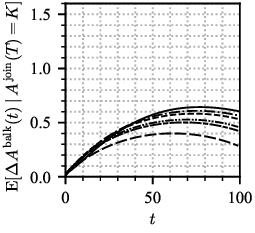}}
\\
\subfigure[\scriptsize \mbox{$\E[\varDelta A(t) \mid A^\join(T)=K]$}]{%
\includegraphics{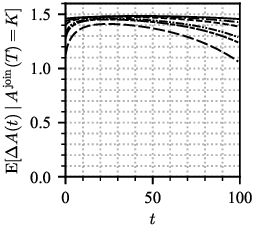}}
\subfigure[\scriptsize \mbox{$\E[\varDelta D(t) \mid A^\join(T)=K]$}]{%
\includegraphics{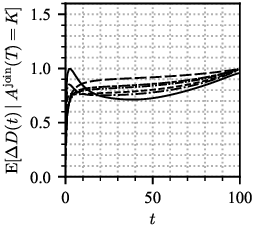}}
\subfigure[\scriptsize \mbox{$\Pr[L(t) \geq 1 \mid A^\join(T)=K]$}]{%
\includegraphics{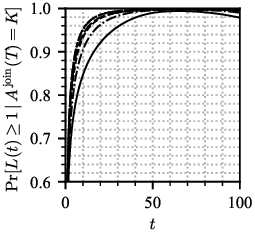}}
\caption{%
Transient behavior of several performance metrics for different values
of the coefficient of variation $C_V$ of service times 
($\lambda = 1.5$ and $\beta_\ell$ is given by \eqref{eq:beta}).}
\label{fig:cv}
\end{figure}

We next examine the effect of the coefficient of variation $C_V$ of
service times on the transient behavior of the system. 
Figure~\ref{fig:cv} shows the dynamics of various performance metrics
for several values of $C_V$, with the arrival rate fixed at $\lambda = 1.5$.
From Figure~\ref{fig:cv} (a), we observe that $\E[L(t) \mid A^{\join}(T)
=K]$ increases initially and then decreases after reaching a peak,
regardless of $C_V$. As $C_V$ increases, this peak becomes more
pronounced and occurs later in time. 
Figure~\ref{fig:cv} (a) and (c) show that the timing of balking events
remains closely aligned with the level of system congestion, as
reflected in $\E[L(t) \mid A^{\join}(T)=K]$, across all cases.

Figure~\ref{fig:cv} (d) and (e) show that a higher $C_V$ leads to a
stronger bias in the number of departures: 
for large values of $C_V$, the departure rate $\E[\varDelta D(t) \mid
A^{\join}(T)=K]$ tends to be elevated during the early period,
while the arrival rate $\E[\varDelta A(t) \mid A^{\join}(T)=K]$
remains relatively stable.
In contrast, when $C_V$ is small, the bias shifts to the arrival
process, with $\E[\varDelta A(t) \mid A^{\join}(T)=K]$ exhibiting a
notable time dependence. In this case, the departure rate gradually
approaches one, reflecting saturation of system capacity.

In this setting, the unconditional expected number of arrivals is
$\lambda T = 150$, which exceeds the fixed number of joined customers
$K = 100$ by a wide margin.  When $C_V$ is large, frequent occurrences
of long service times lead to heavy congestion, resulting in more
balking events.  In contrast, when $C_V$ is small, such long service
times are rare, and the relative variability in the arrival process
becomes more influential, shifting the bias toward fluctuations in
arrivals.

Finally, we examine how different forms of the entry probability
$\beta_\ell$ affect system behavior. In addition to the baseline
defined in \eqref{eq:beta}, we consider the following four alternative
shapes:
\begin{alignat}{2}
\beta_\ell
&=
\max\left((1-0.04\ell)\exp(0.04\ell), 0\right),
&&\quad \ell = 0,1,\ldots,100,
\label{eq:beta2}
\\
\beta_\ell
&=
\max\left((1-0.03\ell)\exp(0.017\ell), 0\right),
&&\quad \ell = 0,1,\ldots,100,
\label{eq:beta3}
\\
\beta_\ell
&=
0.4+0.6(1-0.02\ell)\exp(-0.025\ell),
&&\quad \ell = 0,1,\ldots,100,
\label{eq:beta4}
\\
\beta_\ell
&=
0.6+0.4(1-0.015\ell)\exp(-0.09\ell),
&&\quad \ell = 0,1,\ldots,100.
\label{eq:beta5}
\end{alignat}
These choices are calibrated so that, without condition that the
total number of joined customers is given with $\lambda = 1.5$ and $C_V = 1.0$, 
the expected number of balking customers in
$[0, T]$ satisfies $36 < \E[A^\balk(T)] < 37$.  
The corresponding entry probability functions are plotted in
Figure~\ref{fig:beta-2}.

\begin{figure}[!t]
\centering
\includegraphics{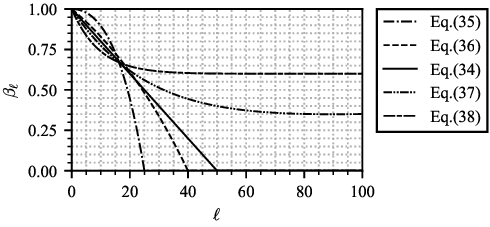}
\caption{Form of the entry probability 
$\beta_\ell$ given by \eqref{eq:beta}--\eqref{eq:beta5}.}
\label{fig:beta-2}
\end{figure}

\begin{figure}[!t]
\centering
\subfigure{%
\fbox{
\includegraphics{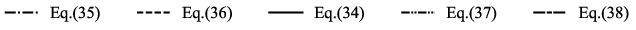}}}
\\
\setcounter{subfigure}{0}
\subfigure[\scriptsize \mbox{$\E[L(t) \mid A^\join(T)=K]$}]{%
\includegraphics{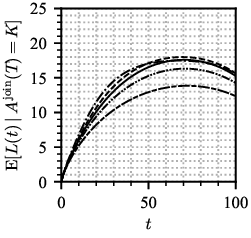}}
\subfigure[\scriptsize \mbox{$\E[\varDelta A^\join(t) \mid A^\join(T)=K]$}]{%
\includegraphics{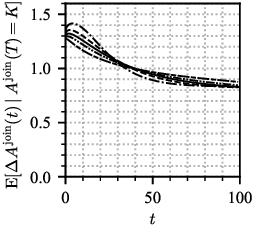}}
\subfigure[\scriptsize \mbox{$\E[\varDelta A^\balk(t) \mid A^\join(T)=K]$}]{%
\includegraphics{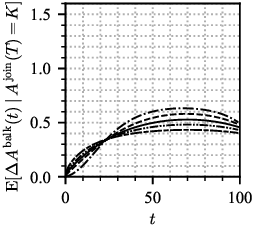}}
\\
\subfigure[\scriptsize \mbox{$\E[\varDelta A(t) \mid A^\join(T)=K]$}]{%
\includegraphics{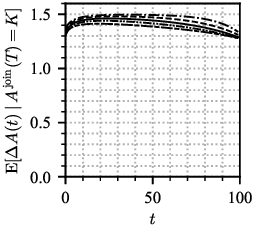}}
\subfigure[\scriptsize \mbox{$\E[\varDelta D(t) \mid A^\join(T)=K]$}]{%
\includegraphics{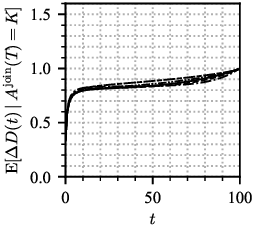}}
\subfigure[\scriptsize \mbox{$\Pr[L(t) \geq 1 \mid A^\join(T)=K]$}]{%
\includegraphics{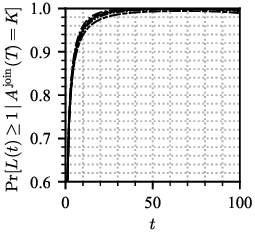}}
\caption{%
Transient behavior of several performance metrics for different forms
\eqref{eq:beta}--\eqref{eq:beta5} of the entry probability $\beta_\ell$ 
($\lambda=1.5$ and $C_V=1.0$).}
\label{fig:beta}
\end{figure}

Figure~\ref{fig:beta} presents the transient behavior of the system under  
the five different entry probability functions $\beta_\ell$, with  
the arrival rate $\lambda=1.5$ and the coefficient of variation of 
service times $C_V=1.0$.
As shown in Figure~\ref{fig:beta} (a), the mean number of customers  
in the system exhibits a unimodal shape in all cases, and
the peak tends to be lower when $\beta_\ell$ is convex.
Figure~\ref{fig:beta} (b) and (c) show that when $\beta_\ell$
is concave, balking tends to occur more frequently in the later part  
of the interval, suggesting that customer entries are more
concentrated early on.
From Figure~\ref{fig:beta} (c), (d), and (e), we observe that the 
convex $\beta_\ell$ leads to fewer arrivals and more frequent  
departures, with fewer number of balking customers.

These observations can be explained by the structure of the entry
probability $\beta_{\ell}$. When $\beta_{\ell}$ is concave,
balking increases sharply with congestion, leading to a greater
imbalance in the timing of entries.  
In contrast, the convex $\beta_{\ell}$ results in a more gradual
response to congestion, making it less likely that the system
compensates by increasing congestion to induce balking; instead,
the arrival rate decreases and the departure rate becomes more similar
to that in the unconditioned system.

\section{Conclusion}

In this paper, we analyzed a Markovian queueing system with balking,
where arrivals follow a nonhomogeneous Poisson process over a finite
time interval $[0,T]$, and customers may balk depending on the number
of customers present upon arrival. Our analysis focused on the system
behavior conditioned on exactly $K$ customers having joined the system
during the interval.

We addressed two main computational challenges arising in the analysis
of the cumulative number of balking events under the fixed-$K$
setting. First, directly analyzing the extended Markov chain leads to an intractably
large state space. To overcome this, we introduced a PGF
representation and derived a first-order
differential equation satisfied by the unconditioned PGF. 
This formulation enables efficient computation without explicitly
tracking the number of balking customers.

Second, to avoid computing the transition probabilities
$\Pr[A^\join(T)=K \mid A^\join(t)=k, L(t)=\ell, S(t)=s]$ repeatedly for multiple
time points, we considered a backward-in-time approach. We showed that
the conditional probability $\Pr[A^\join(T) = K \mid A^\join(T-u) = k,
L(T-u) = \ell, S(T-u) = s]$ satisfies a linear differential equation
in $u$, allowing efficient evaluation across a range of values.

For the case where system parameters are piecewise constant, we
developed a numerical procedure to compute the conditional factorial
moments. Numerical examples for an M/PH/1 queue illustrated that
conditioning on the total number of joined customers induces
time-dependent dynamics that differ markedly from typical steady-state
behavior. For instance, the expected number of customers in the system
may increase and then decrease even under a constant arrival rate.
We showed that such behaviors are driven by nontrivial biases in both
arrivals and departures caused by the conditioning.

In this paper, we assumed that once customers join the system, they
remain until service completion. Under this assumption, the total
number of customers served equals the number of joined customers.
However, in many real-world systems, customers may renege before
completing service. In such settings, conditioning on the number of
served customers no longer coincides with conditioning on the number
of entries. Extending our approach to such cases, where early
departures occur, remains an important direction for future work.

\section*{Acknowledgments}
This work was supported in part by JSPS KAKENHI Grant Number JP24K14839.

\appendix

\section{Transition rate matrix and evaluation of 
\eqref{eq:x^0(m)} and \eqref{eq:x^i(m)} in the M/PH/1 queue}\label{sec:M/PH/1}
In this section, we present the specific 
form of the northwest corner block of the 
transition rate matrix in an M/PH/1 queue 
and describe the concrete calculation method 
for $\bm{x}_n^{(i)}(m)(\bm{P}_n-\theta_n^{-1}\lambda_n\bm{I})$ 
appearing on the right-hand side of equations 
\eqref{eq:x^0(m)} and \eqref{eq:x^i(m)}. 
For the sake of notational simplicity, 
we omit subscripts and define:
\[
\lambda_1 \eqqcolon \lambda, 
\qquad 
\bm{Q}_1 \eqqcolon \bm{Q}.
\]
Additionally, we define:
\[
\theta \ceq \max_i|[\bm{Q}]_{i,i}|+\lambda,
\qquad 
\bm{P} \ceq \bm{I}+\theta^{-1}\bm{Q}.
\]
The service time distribution follows a 
phase-type distribution characterized by 
the initial state $\bm{\gamma}$ and the 
transition rate matrix $\bm{\Gamma}$. 
The initial state $\bm{\gamma}$ and the 
transition rate matrix $\bm{\Gamma}$ for 
given mean and coefficient of variation 
can be obtained from \cite[pp.~353, 358--359]{Tijms1994}.

Let $\bm{L}(t) \ceq (L(t), S(t))$ denote 
the number of customers in the system and 
the service phase. The northwest corner 
block $\bm{Q}$ of the transition rate matrix 
of the continuous-time Markov chain 
$(A^\join(t), \bm{L}(t))_{t \geq 0}$ is then given by
\[
\bm{Q}
=
\left[
\begin{array}{cccccc}
\bm{X}_0 & \bm{Y}_0 & \bm{O} & \cdots & \bm{O} & \bm{O}\\
\bm{O} & \bm{X}_1 & \bm{Y}_1 & \cdots & \bm{O} & \bm{O}\\
\bm{O} & \bm{O} & \bm{X}_2 & \cdots & \bm{O} & \bm{O}\\
\vdots & \vdots & \vdots & \ddots & \vdots & \vdots\\
\bm{O} &\bm{O} &\bm{O} & \cdots & \bm{X}_{K-1} & \bm{Y}_{K-1}\\
\bm{O} &\bm{O} &\bm{O} & \cdots & \bm{O} & \bm{X}_{K}
\end{array}
\right].
\]
where $\bm{X}_k$ and $\bm{Y}_k$ (for $k=0,1,\ldots$) 
are given by
\begin{align*}
\bm{X}_k
&=
\red{%
\left[
\begin{array}{cccccc}
-\lambda\beta_0 & \bm{0} & \bm{0} & \cdots & \bm{0} & \bm{0}\\
-\bm{\Gamma}\bm{e} & -\lambda\beta_1\bm{I}+\bm{\Gamma} & \bm{O} & \cdots & \bm{O} & \bm{O}\\
\bm{0} & -\bm{\Gamma}\bm{e}\bm{\gamma} & -\lambda\beta_2\bm{I}+\bm{\Gamma} & \cdots & \bm{O} & \bm{O}\\
\vdots & \vdots & \vdots & \ddots & \vdots & \vdots\\
\bm{0} & \bm{O} & \bm{O} & \cdots & -\lambda\beta_{k-1}\bm{I}+\bm{\Gamma} & \bm{O}\\
\bm{0} & \bm{O} & \bm{O} & \cdots & -\bm{\Gamma}\bm{e} & -\lambda\beta_k\bm{I}+\bm{\Gamma}
\end{array}
\right]},
\\
\bm{Y}_k
&=
\left[
\begin{array}{cccccc}
0 & \lambda\beta_0\bm{\gamma} & \bm{0} & \bm{0} & \cdots & \bm{0}\\
\bm{0} & \bm{O} & \lambda\beta_1\bm{I} & \bm{O} & \cdots & \bm{O}\\
\bm{0} & \bm{O} & \bm{O} & \lambda\beta_2\bm{I} & \cdots & \bm{O}\\
\vdots & \vdots & \vdots & \vdots & \ddots & \vdots\\
\bm{0} & \bm{O} & \bm{O} & \bm{O} & \cdots & \lambda\beta_k\bm{I}
\end{array}
\right].
\end{align*}

Due to the sparsity of $\bm{Q}$, the computation 
of the matrix product $\bm{P}-\theta^{-1}\lambda\bm{I}$ 
from the right can be optimized as follows.  
Given a row vector $\bm{y}$, let 
$\bm{y}' \ceq \bm{y}(\bm{P}-\theta^{-1}\lambda\bm{I})$. 
The elements $\bm{y}_{k,\ell}'$ of $\bm{y}'$, 
corresponding to the state where the cumulative 
number of joins is $k$ and the number of 
customers in the system is $\ell$, can be 
computed using the elements $\bm{y}_{k,\ell}$ of 
$\bm{y}$ as follows:
\begin{align*}
y_{0,0}'
&=
y_{0,0} (1 - \theta^{-1}\lambda(\beta_0 + 1)),
\hspace{10.82em}
k = 0,\ \ell = 0,
\\
y_{k,0}'
&=
y_{k,0} (1 - \theta^{-1}\lambda(\beta_0 + 1))
+
\bm{y}_{k,1} \theta^{-1}(-\bm{\Gamma}\bm{e}),
\hspace{3.67em}
k = 1,2,\ldots,K,\ \ell = 0,
\\
\bm{y}_{1,1}'
&=
y_{0,0} \theta^{-1}\lambda\beta_0 \bm{\gamma}
+
\bm{y}_{1,1} (\bm{I} - \theta^{-1}(\lambda(\beta_1 + 1)\bm{I} - \bm{\Gamma})),
\hspace{1em}
k = 1,\ \ell = 1,
\\
\bm{y}_{k,1}'
&=
y_{k-1,0} \theta^{-1}\lambda\beta_0 \bm{\gamma}
+
\bm{y}_{k,1} (\bm{I} - \theta^{-1}(\lambda(\beta_1 + 1)\bm{I} - \bm{\Gamma}))
+
\bm{y}_{k,2} \theta^{-1}(-\bm{\Gamma}\bm{e})\bm{\gamma},
\\ &\hspace{17em}
k = 2,3,\ldots,K,\ \ell = 1,
\\
\bm{y}_{k,k}'
&=
\bm{y}_{k-1,k-1} \theta^{-1}\lambda\beta_{k-1} \red{\bm{I}}
+
\bm{y}_{k,k} (\bm{I} - \theta^{-1}(\lambda(\beta_k + 1)\bm{I} - \bm{\Gamma})),
\\ &\hspace{17em}
k = 2,3,\ldots,K,\ \ell = k,
\\
\bm{y}_{k,\ell}'
&=
\bm{y}_{k-1,\ell-1} \theta^{-1}\lambda\beta_{\ell-1} \red{\bm{I}}
+
\bm{y}_{k,\ell} (\bm{I} - \theta^{-1}(\lambda(\beta_\ell + 1)\bm{I} - \bm{\Gamma}))
+
\bm{y}_{k,\ell+1} \theta^{-1}(-\bm{\Gamma}\bm{e})\bm{\gamma},
\\ &\hspace{17em}
k = 2,3,\ldots,K,\ \ell = 2,3,\ldots,k-1.
\end{align*}

$\bm{x}_n^{(i)}(m)(\bm{P}_n-\theta_n^{-1}\lambda_n\bm{I})$ 
appearing on the right-hand side of 
\eqref{eq:x^0(m)} and \eqref{eq:x^i(m)} can be 
computed using the above formulas. Moreover, 
for $\bm{P}_n \bm{q}(T-T_{N-n+1},K)$ in \eqref{eq:q-uniformization}, 
a similar computation can be performed 
utilizing the sparsity of $\bm{Q}$.

\end{document}